\documentclass[12pt,a4paper,twoside]{amsart}
\usepackage{amscd}
\usepackage{amssymb}

\newtheorem{theorem}{Theorem}[section] 

\theoremstyle{definition}

\newtheorem{defn}{Definition}[section]

\newtheorem{lemma}[defn]{Lemma}
\newtheorem{proposition}[defn]{Proposition}
\newtheorem{cor}[defn]{Corollary}
\newtheorem{notation}[defn]{Notation}

\theoremstyle{remark}
\newtheorem*{remark}{Remark}
\newtheorem*{remarks}{Remarks}
\usepackage{amsfonts}
\newcommand{\field}[1]{\mathbb{#1}}
\newcommand{\Q}{\field{Q}}

\newcommand{\Z}{\field{Z}}
\newcommand{\C}{\field{C}}

\newcommand{\D}{\Delta}

\newcommand{\g}{\gamma}
\newcommand{\G}{\Gamma}

\newcommand{\ra}{\rightarrow}

\usepackage{amsmath}
\usepackage{amssymb}
\usepackage{amsfonts}
\usepackage{enumerate}

\numberwithin{equation}{section}

\date{}

\def\vp{\varphi}

\def\Ker{\text{\rm Ker}}
\def\PSL{\text{\rm PSL}}
\def\SL{\text{\rm SL}}
\def\GL{\text{\rm GL}}
\def\Mod{\text {Mod}}
\def\Aut{\text{Aut}}

\def\Z{\mathbb  Z}

\def\Ga{\Gamma}
\def\bbr{\mathbb{F}}
\def\bbf{\mathbb{R}}
\def\bbz{\mathbb{Z}}
\def\bbz{\mathbb{Q}}

\theoremstyle{plain}
\newtheorem{thm}{Theorem}[section]
\newtheorem*{thm*}{Theorem}
\newtheorem{corr}[thm]{Corollary}
\newtheorem{prop}[thm]{Proposition}
\newtheorem*{prop*}{Proposition}
\newtheorem*{prop**}{\ }
\def\beq{\begin{equation}}

  \def\ee{\end{equation}}

\theoremstyle{definition} 
\newtheorem{definition}[thm]{Definition}
\newtheorem*{definition*}{Definition}
\newtheorem{question}[thm]{Question}

\newtheorem*{thm1*}{Theorem A1}
\newtheorem*{thm2*}{Theorem A2}

\newtheorem*{conjecture*}{Conjecture}

\newtheorem*{claim*}{Claim}

\newtheorem*{remark*}{Remark}
\def\bbz{\mathbb{Z}}
\def\bbq{\mathbb{Q}}
\def\bbf{\mathbb{F}}
\def\bbr{\mathbb{R}}
\def\bba{\mathbb{A}}
\def\bbc{\mathbb{C}}

\def\be{\begin{equation}}
\def\ee{\end{equation}}

\theoremstyle{remark}  
\begin{document}

\title[Congruence Topology]{The Congruence Topology, \\
Grothendieck Duality and Thin groups}
\author[A. Lubotzky]{Alexander Lubotzky}
\address{Institute of Mathematics\\
Hebrew University\\
Jerusalem 9190401, Israel\\
alex.lubotzky@mail.huji.ac.il}
\author[T.N. Venkataramana]{T.N. Venkataramana}
\address{Tata Institute of Fundamental Research\\
Homi Bhabha Road\\
Colaba, Mumbai 400005, India\\
venky@math.tifr.res.in}
\maketitle


\baselineskip 16pt

\begin{abstract}

This paper answers a question raised by Grothendieck in 1970 on the ``Grothendieck closure" of an integral linear group and proves a conjecture of the first author made in 1980. This is done by a detailed study of the congruence topology of arithmetic groups, obtaining along the way, an arithmetic analogue of a classical result of Chevalley for complex algebraic groups. As an application we also deduce a group theoretic characterization of thin subgroups of arithmetic groups.
\end{abstract}

\section*{0. Introduction}

If  $\varphi: G_1 \to  G_2$ is  a polynomial  map between  two complex
varieties, then  in general  the image of  a Zariski closed  subset of
$G_1$ is not  necessarily closed in $G_2$.  But here is  a classical result:

\begin{thm*}[Chevalley] \label{Chevalleythm} 
If  $\varphi$ is a  polynomial homomorphism
between two  complex algebraic groups  then $\varphi(H)$ is  closed in
$G_2$ for every closed subgroup $H$ of $G_1$.
\end{thm*}

There is an arithmetic analogue of this issue:
Let $G$ be a $\bbq$-algebraic group, let $\bba_f = \Pi^*_{p\, prime} \bbq_p$ be the ring of finite \'adeles over $\bbq$.  The topology of $G(\bba_f)$ induces the congruence topology on $G(\bbq)$.  If $K$ is compact open subgroup of $G(\bba_f)$ then $\Ga = K \cap G(\bbq)$ is called a congruence subgroup of $G(\bbq)$.   This defines the congruence topology   on $G(\bbq)$ and on  all its subgroups.
A subgroup of $G(\bbq)$ which is closed in this topology is called congruence closed.
A subgroup $\Delta$  of $G$ commensurable to $\Ga$ is called an arithmetic group.

Now, if $\varphi:G_1 \to G_2$ is a $\bbq$-morphism between two $\bbq$-groups, which is a surjective homomorphism (as $\bbc$-algebraic groups) then the image of an arithmetic subgroup $\Delta$ of $G_1$ is an arithmetic subgroup of $G_2$ (\cite[Theorem 4.1 p.~174]{Pl-Ra}), but the image of a congruence subgroup is not necessarily a congruence subgroup.  It  is  well known  that
$\SL_n(\bbz)$ has  congruence subgroups whose images  under the adjoint
map  $\SL_n(\bbz) \to  \PSL_n(\bbz) \hookrightarrow  \Aut (M_n(\bbz))$
are   not    congruence   subgroups   (see    \cite{Ser}   and   Proposition \ref{congruence image}     below   for   an    exposition   and
explanation).  So, the direct analogue of Chevalley theorem does not hold.   Still,  in this  case,  if  $\Gamma$  is a  congruence
subgroup    of    $\SL_n(\bbz)$,    then    $\varphi(\Ga)$    is    a    normal subgroup of
$\overline{\varphi(\Ga)}$, the  (congruence) closure of $\varphi(\Ga)$
in $\PSL_n(\bbz)$,  and the quotient  is a finite abelian  group.  Our
first technical  result says that the  general case is  similar. It is
especially important  for us that  when $G_2$ is simply  connected, the
image of a congruence subgroup of $G_1$ is a congruence subgroup in
$G_2$ (see  Proposition \ref{arithmeticchevalley} (ii) below).

Before stating the result, we give the following definition and set some notations for the rest of the paper:

  Let $G$ be a linear algebraic group over $\bbc$, $G^0$ - its connected component, and $R = R(G)$ - its solvable radical, i.e. the largest connected normal solvable subgroup of $G$.  We say that $G$ is \emph{essentially simply connected} if $G_{ss}:= G^0/R$ is simply connected.

Given a subgroup $\Ga$ of $GL_n$, we will  throughout the paper denote by $\Ga^0$ the intersection of $\Ga$ with $G^0$, where $G^0$ is the connected component of $G$ - the Zariski closure of $\Ga$.   Therefore, $\Ga^0$ is always a finite index normal subgroup of $\Ga$.

The notion ``essentially simply connected"  will play an important role in this paper due to the following proposition, which can be considered as the arithmetic analogue of Chevalley's result above:

\begin{prop}\label{arithmeticchevalley}
 \begin{enumerate}[(i)]
\item  If  $\varphi:  G_1  \to  G_2$ is  a  surjective  (over  $\bbc$)
algebraic homomorphism between two $\bbq$-defined algebraic groups,  then for every  congruence
closed  subgroup $\Ga$  of $G_1
(\bbq)$,  the image  $\vp(\Ga^0)$ is  normal in  its  congruence closure
$\overline{\vp(\Ga^0)}$ and $\overline{\vp(\Ga^0)}/\vp  (\Ga^0)$ is a finite
abelian group.
\item  If  $G_2$  is  essentially  simply  connected, and $\Gamma$ a congruence subgroup of $G_1$  then
$\overline{\vp(\Ga)}  = \vp (\Ga)$,  i.e., the  image of  a congruence
subgroup is congruence closed.
\end{enumerate}
\end{prop}

This  analogue of  Chevalley's theorem,  and a  result  of \cite{Nori},
\cite{Weis} enable us to prove:

\begin{prop}\label{congruenceimage}  If $\Ga_1 \le \GL_n(\bbz)$ is a congruence closed subgroup (i.e. closed in the congruence topology) with Zariski closure $G$, then there exists a congruence subgroup $\Gamma$ of $G$, such that $[\Gamma, \Gamma] \le \Ga_1^0 \le \Ga$.
If $G$ is essentially simply connected then the image of $\Ga _1$ in $G/R(G)$ is actually a congruence subgroup.
\end{prop}


We apply Proposition \ref{arithmeticchevalley} (ii)  in two directions:

\begin{enumerate}[(A)]

\item Grothendieck-Tannaka duality for discrete groups, and

\item  A  group  theoretic   characterization  of  thin  subgroups  of
arithmetic groups.
\end{enumerate}
\medskip
\subsection*{Grothendieck  closure}  In \cite{Gro},
Grothendieck was interested in the following question:

\begin{question} \label{isocompletion} Assume  $\varphi: \Ga_1   \to  \Ga_2$   is  a
homomorphism  between  two   finitely  generated  residually finite groups  inducing  an
isomorphism  $\hat\vp:\hat\Gamma_1 \to  \hat\Ga_2$ between  their
profinite completions.  Is $\vp$ already an isomorphism?

\end{question}

 To tackle Question \ref{isocompletion}, he introduced the following notion.  Given a finitely generated
group $\Ga$ and a commutative  ring $A$ with identity, let $Cl_A(\Ga)$
be the  group of all automorphisms  of the forgetful  functor from the
category  $\Mod_A(\Ga)$  of all  finitely  generated $A$-modules  with
$\Ga$  action  to  $\Mod_A(\{   1  \})$,  preserving  tensor  product.
Grothendieck's strategy  was the following: he showed  that, under the
conditions of Question \ref{isocompletion}, $\vp$ induces an isomorphism from $\Mod_A(\Ga_2)$ to
$\Mod_A(\Ga_1)$,   and   hence  also   between   $Cl_A  (\Ga_1)$   and
$Cl_A(\Ga_2)$.  He then asked:
\begin{question} \label{closureisomorphism} Is the  natural  map $\Ga  \hookrightarrow
Cl_{\bbz} (\Ga)$  an isomorphism  for a finitely  generated residually
finite group?
\end{question}

An affirmative  answer to Question \ref{closureisomorphism}  would imply an affirmative  answer to
Question \ref{isocompletion}.   Grothendieck then showed  that arithmetic  groups with  the
(strict)     congruence    subgroup     property    do  indeed    satisfy
$Cl_{\bbz}(\Ga)\simeq \Ga$.

Question 0.4 basically asks whether  $\Ga$ can be recovered from its
category of representations.  In  \cite{Lub}, the first author phrased
this  question in  the  framework  of Tannaka  duality,  which asks  a
similar question for compact Lie groups.  He also gave a more concrete
description of $Cl_\bbz(\Ga)$:
\begin{equation}\label{profinitegrothendieck} Cl_\bbz (\Ga) = \{ g \in \hat \Ga | \hat\rho
(g)   (V)   =  V,\quad \forall  \quad (\rho, V)   \in   \Mod_\bbz
(\Ga)\}.\end{equation}

Here $\hat\rho$  is the continuous  extension $\hat\rho: \hat  \Ga \to
\Aut  (\hat V)$  of the  original representation  $\rho: \Ga  \to \Aut
(V)$.

However, it is also  shown in \cite{Lub}, that  the answer to Question \ref{closureisomorphism}
 is  negative. The  counterexamples provided  there
 are  the arithmetic  groups for  which the  weak  congruence subgroup
property holds but not the  strict one, i.e.  the congruence kernel is
finite  but non-trivial.  It  was conjectured  in \cite[Conj A,
p. 184]{Lub},   that for an arithmetic group $\Ga$, $Cl_\bbz (\Ga) = \Ga$ if and
only if  $\Ga $  has the (strict)  congruence subgroup  property.  The
conjecture was left open even for $\Ga = \SL_2 (\bbz)$.

In  the  almost  40  years  since  \cite{Lub}  was  written  various
counterexamples were given to question \ref{isocompletion}  (\cite{Pl-Ta1}, \cite{Ba-Lu}, \cite{Br-Gr}, \cite{Py})
which  also give    counterexamples to  question \ref{closureisomorphism},  but it  was not  even
settled whether  $Cl_\bbz(F) =  F$ for finitely  generated non-abelian
free groups $F$.

We can  now answer this and,  in fact, prove  the following surprising
result, which gives an essentially  complete answer to Question \ref{closureisomorphism}.

\begin{thm}\label{maintheorem}  Let  $\Ga$ be a finitely  generated subgroup of
$\GL_n(\bbz)$.   Then  $\Ga$  satisfies Grothendieck-Tannaka  duality,
i.e. $Cl_\bbz(\Ga)  = \Ga$ if  and only if  $\Ga $ has  the congruence
subgroup property i.e., for some (and consequently for every) faithful
representation $\G \ra \GL_m(\Z)$ such that the Zariski closure $G$ of $\G$
is essentially simply connected, every finite index subgroup of $\Ga $
is closed in  the congruence topology of $\GL_n(\bbz)$.  In such  a case, the
image of the group $\G$ in the semi-simple (simply connected) quotient
$G/R$ is a congruence arithmetic group.
\end{thm}

The  Theorem  is surprising  as  it shows  that  the  cases proved  by
Grothendieck himself  (which motivated him to suggest that  the duality
holds in  general) are essentially  the only cases where  this duality
holds.

Let us  note that the  assumption on $G$ is not really restrictive.
In Lemma  \ref{simplyconnectedsaturate}, we  show that for  every $\Ga
\le \GL_n(\bbz)$ we can find  an ``over" representation of $\Ga$ into
$\GL_m  (\bbz)$ (for some  $m$) whose  Zariski closure  is essentially
simply connected.

Theorem \ref{maintheorem}  implies Conjecture A of [Lub].

\begin{corr}\label{lubconjecture}  If $G$  is a  simply connected  semisimple $\bbq$-algebraic group, and  $\Ga$ a congruence
subgroup of  $G(\bbq)$, then $Cl_\bbz (\Ga)  = \Ga$ if and  only if $\Ga$
satisfies the (strict) congruence subgroup property.
\end{corr}

In particular:
\begin{corr}\label{0.10}  $Cl_\bbz(F)   \neq  F$  for   every  finitely
generated  free  group  on   at  least  two  generators;  furthermore,
$Cl_\bbz (\SL_2(\bbz)) \neq \SL_2 (\bbz)$.
\end{corr}  In fact,  it will follow from our results  that   $Cl_\bbz(F)$ is
uncountable.  \\

Before moving on to the last application, let us say a few words about
how Proposition \ref{arithmeticchevalley} helps to  prove a result like Theorem \ref{maintheorem}.
The  description of $Cl_\bbz  (\Ga)$ as in  Equation \ref{profinitegrothendieck}  implies  that
\begin{equation}\label{eq0.2}
Cl_\bbz  (\Ga)   =  \underset{\rho}{\lim\limits_{\leftarrow}}  \quad
\overline{\rho(\Ga)}  \end{equation}
when  the limit  is over  all
$(\rho, V)$  when $V$  is  a  finitely generated  abelian  group, $\rho$  a
representation $\rho:  \Ga \to \Aut (V)$ and  $\overline{\rho (\Ga)} =
\hat\rho (\hat\Ga)\cap \Aut (V) \subseteq  \Aut (\hat V)$.  This is an
inverse limit  of countable discrete groups,  so one can  not say much
about it unless the connecting homomorphisms are surjective, which is,
in general, not the case. Now, $\overline{\rho(\Ga)}$ is the congruence closure of $\rho(\Ga)$ in $\Aut (V)$ and Proposition \ref{arithmeticchevalley} shows that the corresponding maps are ``almost" onto, and  are  even  surjective if  the
modules $V$ are what we call here ``simply connected representations",
namely those cases  when $V$ is torsion free  (and hence isomorphic to
$\bbz^n$ for  some $n$) and  the Zariski closure of  $\rho(\Gamma)$ in $\Aut (\bbc \underset{\bbz}{\otimes} V) =
\GL_n (\bbc)$ is essentially  simply connected.  We show further that
the category $\Mod_{\bbz}(\Ga)$ is  ``saturated" with such modules (see Lemma
\ref{simplyconnectedsaturate})  and  we deduce  that  one can  compute
$Cl_\bbz(\Ga)$  as  in  Equation \ref{profinitegrothendieck}  by considering  only  simply  connected
representations.   We can  then  use Proposition \ref{arithmeticchevalley}(b),  and get  a fairly  good
understanding  of $Cl_\bbz(\Ga)$.   This enables  us to  prove Theorem
\ref{maintheorem}. In addition, we also deduce:
\begin{corr}\label{simplyconnectedonto}  If $(\rho, V)$ is  a simply  connected representation,
then   the  induced   map  $Cl_\bbz(\Ga)   \to  \Aut   (V)$   is  onto $Cl_\rho(\Gamma): =
\overline{\rho(\Ga)}$ - the congruence closure of $\Ga$.
\end{corr}

From Corollary \ref{simplyconnectedonto}  we can deduce our last
application.

\subsection*{Thin groups} In recent years, following \cite{Sar},
there has been a  lot of interest in the  distinction between thin
subgroups and arithmetic subgroups of algebraic groups.  Let us
recall:

\begin{definition}\label{0.12}  A subgroup  $\Ga  \le \GL_n(\bbz)$  is
called {\bf thin} if it is  of infinite index in $G \cap \GL_n(\bbz)$,
when  $G$ is  its Zariski  closure in  $\GL_n$.  For  a  general group
$\Ga$, we  will say that it  is a {\bf thin  group} (or it  {\bf has a
thin representation})  if for some  $n$ there exists  a representation
$\rho:\Ga \to \GL_n(\bbz)$ for which $\rho(\Ga)$ is thin.
\end{definition}

During the last five decades a lot of attention was given to the study
of  arithmetic groups,  with  many remarkable  results, especially  for
those of higher rank (cf.  \cite{Mar}, \cite{Pl-Ra} and the references
therein). Much less  is known about thin groups.  For  example, it is not
known if there exists a thin group with property $(T)$.  Also, given a
subgroup of an arithmetic group (say, given by a set of generators) it
is  difficult to decide  whether it  is thin  or arithmetic  (i.e., of
finite or infinite index in its integral Zariski closure).

It  is therefore  of interest  and  perhaps even  surprising that  our
results enable us to  give a purely group theoretical characterization
of thin  groups $\Ga \subset GL_n(\bbz)$.  Before  stating the precise  result, we make the topology on $Cl_\bbz(\Ga)$ explicit.
             If we take the class of simply connected representations $(\rho,V)$ for computing the group $Cl_\bbz(\Ga)$, one can then show that
$Cl_\bbz(\Ga)/\Ga$ is a {\it closed} subspace of the product $\prod _\rho (Cl_\rho(\Ga)/\Ga)$, where each $Cl_\rho(\Ga)/\Ga$ is given the discrete topology.  This is the topology on the quotient space $Cl_\bbz(\Ga)/\Ga$ in the following theorem. We can now state:

\begin{thm}\label{thincriterion} Let  $\Ga$ be finitely  generated $\bbz$-linear group.  Then
$\Ga$ is a  thin group if and  only if it satisfies (at  least) one of
the following conditions:
\begin{enumerate}
\item $\Ga$ is not $FAb$ (namely, it does have a finite index subgroup
with an infinite abelianization), or
\item $Cl_\bbz (\Ga)/\Ga $ is not compact.
\end{enumerate}
\end{thm}

\medskip
\noindent {\bf Warning} \ There are groups $\Ga$ which can be realized
both as  arithmetic groups as well  as thin groups.   For example, the
free group is an arithmetic subgroup of $\SL_2(\bbz)$, but at the same
time a thin subgroup of every semisimple group, by a well known result
of Tits \cite{Ti}.  In our terminology this  is a thin group.

\medskip

T.N.V. thanks the Math Department of the Hebrew University for great hospitality while a major part of this work was done.  He would also like to thank the JC Bose fellowship (SR/S2/JCB-22/2017) for support during the period 2013-2018.

\smallskip

The authors thank the Math Department of the University of Marseilles, and the conference at Oberwolfach, where the work was completed.  We  would especially like to thank Bertrand Remy for many interesting discussions and for his warm hospitality.

\smallskip

A.L. is indebted to ERC, NSF and BSF for support.

\section{Preliminaries on Algebraic Groups over $\Q$}

We recall the definition of an essentially simply connected group:
\begin{defn}  Let $G$  be  a  linear algebraic  group  over $\C$  with
maximal connected  normal solvable subgroup $R$ (i.e.  the radical
of  $G$)  and identity  component  $G^0$.  We  say  that  $G$ is  {\bf
essentially  simply connected}  if  the semi-simple  part  $G^0/R=H$
is  a simply connected.
\end{defn}

Note that $G$ is essentially  simply connected if and only  if,  the quotient
$G^0/U$ of the  group $G^0$ by its unipotent radical  $U$ is a product
$H_{ss}\times S$ with $H_{ss}$ simply connected and semi-simple,
and $S$ is a torus.

For  example,  a semi-simple  connected  group  is essentially  simply
connected  if  and   only  if  it  is  simply   connected.  The  group
$\mathbb{G}_m\times  \SL_n$ is  essentially simply  connected; however,
the  radical of  the group  $\GL_n$  is the  group $R$  of scalars  and
$\GL_n/R=\SL_n/centre$, so $\GL_n$   is {\it not}  essentially simply  connected.  We
will show later (Lemma \ref{surjectivemorphisms}(iii))
that every  group has a  finite cover which  is essentially
simply connected.

\begin{lemma}   \label{essentiallysimplyconnected}  Suppose  $G\subset
G_1\times G_2$  is a subgroup of  a product of  two essentially simply
connected linear  algebraic groups  $G_1,G_2$ over $\C$;  suppose that
the  projection   $\pi  _i$  of   $G$  to  $G_i$  is   surjective  for
$i=1,2$. Then $G$ is also essentially simply connected.
\end{lemma}

\begin{proof} Assume, as we may, that $G$ is connected. Let $R$ be the radical  of $G$. The projection of $R$ to
$G_i$   is   normal   in    $G_i$   since   $\pi_i:   G\ra   G_i$   is
surjective. Moreover, $G_i/\pi _i(R)$  is the image of the semi-simple
group $G/R$; the latter has a Zariski dense compact subgroup, hence so
does $G_i/\pi  _i(R)$; therefore, $G_i/\pi _i(R)$ is  reductive and is
its own  commutator.  Hence $G_i/\pi  _i(R)$ is semi-simple  and hence
$\pi _i(R)=R_i$ where $R_i$ is the radical of $G_i$.
Let $R^* = G \cap (R_1 \times R_2)$. Since $R_1 \times R_2$ is the radical of $G_1 \times G_2$, it follows that $R^*$ is a solvable normal subgroup of $G$ and hence its connected component is contained in $R$.  Since $R \subseteq R_1 \times R_2$, it follows that $R$ is precisely the connected component of the identity of $R^*$.
We then  have the inclusion $G/R^*\subset G_1/R_1\times
G_2/R_2$ with projections again being surjective.

By   assumption,    each   $G_i/R_i=H_i$   is    semi-simple,   simply
connected.    Moreover    $G/R^*=H$    where   $H$    is    connected,
semi-simple. Thus we have the inclusion $H\subset H_1\times H_2$. Now,
$H\subset H_1\times H_2$ is such  that the projections of $H$ to $H_i$
are surjective,  and each  $H_i$ is simply  connected. Let $K$  be the
kernel of the  map $H\ra H_1$ and $K^0$  its identity component.  Then
$H/K^0 \ra H_1$ is a surjective map of connected algebraic groups with
finite  kernel. The simple  connectedness of  $H_1$ then  implies that
$H/K^0=H_1$ and  hence that $K=K^0 \subset \{1\}\times  H_2$ is normal
in $H_2$.

Write  $H_2=F_1\times \cdots  \times  F_t$ where  each  $F_i$ is  {\it
simple} and  simply connected.  Now, $K$ being a closed normal subgroup of $H_2$ must be equal to $\prod  _{i\in  X} F_i$ for some subset $X$ of $\{ 1, \cdots, t\}$,
 and is  simply
connected.  Therefore, $K=K^0$ is simply connected.

From the preceding two paragraphs, we have that both $H/K$ and $K$ are
simply connected, and hence so is  $H = G/R^*$. Since $R$ is the connected component of $R^*$ and $G/R^*$ is simply connected, it follows that $G/R = G/R^*$ and hence $G/R$ is simply connected.  This completes the proof of the
lemma.
\end{proof}

\subsection{Arithmetic          Groups          and         Congruence
Subgroups} \label{congruence}

In the introduction, we defined the notion of arithmetic and congruence subgroup of $G(\Q)$ using the adelic language. One can define the notion of arithmetic (res. congruence) group in more concrete terms as follows. Given a linear  algebraic group $G\subset \SL_n$ defined  over $\Q$, we
will say that a subgroup  $\Gamma \subset G(\Q)$ is an {\it arithmetic
group}  if is commensurable to $G\cap \SL_n(\bbz)=G(\bbz)$; that is, the
intersection $\Gamma  \cap G(\bbz)$ has finite index  both in $\Gamma$
and in  $G(\bbz)$. It is well  known that the notion  of an arithmetic
group  does not  depend on  the specific  linear  embedding $G\subset
\SL_n$. As in \cite{Ser}, we may define the {\it arithmetic completion}
$\widehat{G}$ of $G(\Q)$  as the completion of the  group $G(\Q)$ with
respect to the topology on $G(\Q)$ as a topological group, obtained by
designating   arithmetic   groups   as   a  fundamental   systems   of
neighbourhoods of identity in $G(\Q)$.

Given $G\subset \SL_n$ as in  the preceding paragraph, we will say that
an  arithmetic  group $\Gamma  \subset  G(\Q)$  is  a {\it  congruence
subgroup} if  there exists an integer  $m \geq 2$ such  that $\Gamma $
contains  the ``principal congruence  subgroup'' $G(m\bbz)=\SL_n(m\bbz)\cap
G$  where  $\SL_n(m\bbz)$  is  the   kernel  to  the  residue  class  map
$\SL_n(\bbz)\ra \SL_n(\bbz/m\bbz)$. We then get the structure of a topological
group  on the  group $G(\Q)$  by designating  congruence  subgroups of
$G(\Q)$  as a fundamental  system of  neighbourhoods of  identity. The
completion  of  $G(\Q)$ with  respect  to  this  topology, is  denoted
$\overline{G}$. Again,  the notion of  a congruence subgroup  does not
depend on the specific linear embedding $G\ra \SL_n$ .

Since every congruence subgroup is an arithmetic group, there exists a
map from $\pi: \widehat{G}\ra \overline{G}$ which is easily seen to be
surjective,  and the  kernel $C(G)$  of $\pi$  is a  compact profinite
subgroup  of  $\widehat{G}$.   This  is  called  the  {\it  congruence
subgroup  kernel}.  One  says  that $G(\Q)$  has  the {\it  congruence
subgroup property}  if $C(G)$  is trivial. This  is easily seen  to be
equivalent to the statement  that every arithmetic subgroup of $G(\Q)$
is a congruence subgroup.

It is  known (see  p.~108, last but  one paragraph of
\cite{Ra2} or  \cite{Ch}) that  solvable groups  $G$ have  the  congruence subgroup
property.

Moreover, every solvable subgroup of $\GL_n(\bbz)$ is polycyclic.  In such a group, every subgroup is intersection of finite index subgroups.  So every solvable subgroup of an arithmetic group is congruence closed.
We will use these facts  frequently in the sequel.

 Another (equivalent) way of  viewing the congruence completion is (see
\cite{Ser}, p.~276, Remarque) as follows: let $\bba _f$ be the ring of
finite adeles  over $\Q$, equipped  with the standard  adelic topology
and let $\bbz _f \subset \bba  _f$ be the closure of $\bbz$.  Then the
group $G({\mathbb A} _f)$ is also a locally compact group and contains
the  group  $G(\Q)$.   The  congruence  completion  $\overline{G}$  of
$G(\Q)$ may be  viewed as the closure of $G(\Q)$  in  $G({\mathbb A} _f)$.

\begin{lemma}   \label{surjectivemorphisms}  Let  $H,H^*$   be  linear
algebraic groups defined over $\Q$.
\begin{enumerate}[(i)]
\item Suppose  $H^* \ra  H$ is  a surjective  $\Q$-morphism.  Let
$(\rho, W _{\Q})$ be  a representation of $H$ defined  over $\Q$.  Then
there exists a faithful  $\Q$-representation $(\tau, V_{\Q})$ of $H^*$
such that $(\rho,W)$ is a sub-representation of $(\tau, V)$.

\item If  $H^*\ra H$ is a  surjective map defined over  $\Q$ , then
the  image of  an arithmetic  subgroup of $H^*$  under the  map $H^*\ra  H$ is  an
arithmetic subgroup of $H$.

\item If $H$ is connected, then there exists a connected essentially simply
connected algebraic group $H^*$ with a surjective $\bbq$-defined homomorphism
$H^*\to H$ with finite kernel.

\item If  $H^*\ra H$ is  a surjective homomorphism of  algebraic $\Q$-groups
which are essentially simply
connected, then the  image of a congruence subgroup  of $H^*(\Q)$ is a
congruence subgroup of $H(\Q)$.

\end{enumerate}

\end{lemma}

\begin{proof} Let $\theta: H^*\ra  \GL(E)$ be a faithful representation
of the  linear algebraic group $H^*$  defined over $\Q$  and  $\tau =\rho \oplus \theta$ as $H^*$-representation.  Clearly $\tau$ is faithful for
$H^*$ and contains $\rho$. This proves (i).

Part (ii) is the statement of Theorem (4.1) of \cite{Pl-Ra}.

We now prove (iii).    Write $H=R G$ as a
product of  its radical $R$  and a semi-simple group  $G$.  Let
$H^*_{ss}\ra G$ be the  simply connected cover of $G$. Hence
$H^*_{ss}$ acts on $R$ through  $G$, via this covering map. Define
$H^*=R\rtimes H  ^*_{ss}$ as a semi-direct product.   Clearly, the map
$H^*\ra H$ has finite kernel and satisfies the properties of (iii).

To prove (iv), we may assume that $H$ and $H^*$ are connected.
If  $U^*,U$  are  the  unipotent radicals  of  $H^*$  and  $H$,  the
assumptions of (iv) do not change for the quotient groups $H^*/U^*$ and
$H/U$. Moreover, since  $H^*$ is the semi-direct product  of $U^*$ and
$H^*/U^*$ (and  similarly for $H,U$) and  the unipotent $\Q$-algebraic
group $U$ has the congruence subgroup property, it suffices to prove
(iv)  when both $H^*$ and $H$ are reductive.  By assumption, $H^*$ and $ H$ are
essentially  simply  connected;   i.e.  $H^*=H^*_{ss}\times  S^*$  and
$H=H_{ss}\times S$  where $S,S^*$  are tori and  $H^*_{ss},H_{ss}$ are
simply connected semi-simple groups.  Thus we have connected reductive
$\Q$-groups  $H^*,H$ with  a surjective  map such  that  their derived
groups are simply connected  (and semi-simple), and the abelianization
$(H^*)^{ab}$ is a torus (similarly for $H$).

Now, $[H^*,H^*]=H^*_{ss}$ is a  simply connected semi-simple group and
hence it is a product $F_1\times \cdots\times F_s$ of simply connected
$\Q$-simple    algebraic    groups $F_i$.     Being   a    factor    of
$[H^*,H^*]=H^*_{ss}$,  the  group $[H,H]=H_{ss}$  is  a  product of  a
(smaller)  number  of  these  $F_i$'s.   After a  renumbering  of  the
indices, we  may assume that  $H_{ss}$ is a product  $F_1\times \cdots
\times F_r$ for some $r\leq s$  and the map $\pi$ on $H^*_{ss}$ is the
projection to the  first $r$ factors. Hence the  image of a congruence
subgroup of $H^*_{ss}$ is a congruence subgroup of $H_{ss}$.

The tori $S^*,S$ have the  congruence subgroup property by a result of
Chevalley (as already stated at the beginning of this section, this is
true  for  all  solvable  algebraic  groups).  Hence  the  image  of  a
congruence subgroup of $S^*$ is a congruence subgroup of $S$.   We thus need  only prove  that every  subgroup of  the reductive
group $H$  of the form  $\Gamma _1\Gamma _2$, where  $\Gamma _1\subset
H_{ss}$ and  $\Gamma _2\subset  S$ are congruence  subgroups,  is
itself a congruence  subgroup of $H$. We use  the adelic form of
the   congruence topology. Suppose  $K$ is  a  compact  open  subgroup of  the
$H(\bba_f)$ where $\bba _f$ is the ring of finite adeles. The image of
$H(\bbq)\cap K$ under the quotient map $H\ra H^{ab}=S$ is a congruence
subgroup  in the  torus $S$  and hence  $H(\bbq)\cap  K' \subset
(H_{ss}(\bbq )\cap K) (S(\bbq )\cap K)$ for some possibly smaller open
subgroup $K'\subset H(\bba_f)$. This proves (iv). \end{proof}

Note that part (iii) and (iv) prove Proposition \ref{arithmeticchevalley}(ii).

\section{The Arithmetic Chevalley Theorem}

In this section, we prove Proposition \ref{arithmeticchevalley}(i). Assume that $\vp: G_1
\ra G_2$ is a surjective  morphism of $\Q$-algebraic groups. We  are to prove  that $\vp (\Ga^0)$ contains  the commutator
subgroup of  a congruence subgroup  of $G_2(\Q)$ containing it.

Before starting on  the proof, let us note that  in general, the image
of  a congruence  subgroup  of $G_1(\Z)$  under  $\vp$ need  not be  a
congruence  subgroup of  $G_2(\Z)$. The following proposition gives a fairly general
situation when this happens.

\begin{prop} \label{congruence image} Let $\pi : G_1\ra G_2$ be a finite covering of semi-simple algebraic groups defined over $\Q$ with $G_1$ simply connected and $G_2$ not.
Assume $G_1 (\bbq)$ is dense in $G_1(\bba_f)$. Write $K$ for the kernel of $\pi$ and  $K_f$ for the kernel of the map
$G_1(\bba _f)\ra G_2(\bba _f)$.  Let $\Ga $ be a congruence subgroup of $G_1(\Q)$ and $H$ its closure in $G_1(\bba _f)$. Then the image $\pi (\Ga) \subset G_2(\Q)$ is a congruence subgroup if and only if $K H\supset K_f$ .

\end{prop}

Before proving the proposition, let us note that while $K$ is finite, the group $K_f$  is a product of infinitely many finite abelian groups and that $K_f$ is central in $\overline{G_1}$. This implies

\begin{corr} \begin{enumerate} [(i)]

\item There are infinitely many congruence subgroups $\Ga _i$ with $\pi (\Ga _i)$  non-congruence subgroups of unbounded finite index in their  congruence closures $  \overline{\Ga _i}$.

\item For each of these $\Ga = \Ga_i$, the image $\pi (\Ga)$ contains the commutator subgroup $[\overline{\Ga},\overline{\Ga}]$, and is normal in $\overline{\Ga}$ (with abelian quotient).

\end{enumerate}

\end{corr}

We now prove Proposition \ref{congruence image}.

\begin{proof} Let $G_3$ be the image of the rational points of $G_1(\Q)$:
\[G_3=\pi (G_1(\Q))\subset G_2(\Q).\] Define a subgroup $\D$ of $G_3$  to be a {\it quasi-congruence subgroup} if the inverse image $\pi ^{-1}(\D)$ is a congruence subgroup of $G_1(\Q)$. Note that the quasi-congruence subgroups of $G_3 $ are exactly the images of congruence subgroups of $G_1(\Q)$ by $\pi$. It is routine to check that by declaring quasi-congruence subgroups to be open,  we get the structure of a topological group on $G_3$ . This topology is  weaker or equal to the arithmetic topology on $G_3$ . However, it is strictly stronger than the congruence topology on $G_3$. The last assertion follows from the fact that the completion of $G_3=G_1(\Q)/K(\Q)$ is the quotient $\overline{G_1}/K$ where $\overline{G_1}$ is the congruence completion of $G_1 (\Q)$, whereas the completion of $G_3$ with respect to the congruence topology is $\overline{G_1}/K_f$.

Now let $\Ga \subset G_1(\Q)$ be a congruence subgroup and $\D_1=\pi (\Ga)$; let $\D_2$ be its congruence closure in $G_3$. Then both $\D_1$ and $\D_2$ are open in the quasi-congruence topology on $G_3$. Denote by $G_3 ^*$ the completion of $G_3$ with respect to the quasi-congruence topology, so $G^*_3 = \overline{G_1}/K$ and denote by $\D_1^*,\D_2^*$ the closures of $\D_1,\D_2$ in $G_3^*$. We then have the equalities
\[\D_2/\D_1= \D_2^*/\D_1^*, \quad \D_2^*= \D^*_1 K_f/K. \]

 Hence $\D_1 ^*=\D_2 ^*$ if and only if $K\D_1^*\supset K_f$. This proves  Proposition \ref{congruence image}.

The proof shows that $\D_1^*$ is normal in $\D_2^*$ (since $ K_f$ is central) with abelian quotient. The same is true for $\D_1$ in $\D_2$ and  the corollary is also proved.
\end{proof}

To continue with the proof of Proposition \ref{arithmeticchevalley}, assume, as we may
(by replacing $G_1$ with the Zariski closure of $\Ga$), that $G_1$ has
no  characters defined  over  $\Q$.  For, suppose  that  $G_1$ is  the
Zariski closure of $\Ga \subset  G_1(\Z)$. Let $\chi :G_1 \ra {\mathbb
G}_m$ be a non-trivial (and therefore surjective) homomorphism defined
over  $\Q$;  then the  image  of  the  arithmetic group  $G_1(\Z)$  in
${\mathbb G}_m(\Q)$ is a Zariski dense arithmetic group.  However, the
only arithmetic groups in ${\mathbb G}_m(\Q)$ are finite and cannot be
Zariski  dense  in ${\mathbb  G}_m$.   Therefore,  $\chi  $ cannot  be
non-trivial. We can also assume that $G_1$ is connected.

We start by proving Proposition 0.1 for the case that $\Ga $ is a congruence subgroup.

If we write $G_1=R _1H_1$ where  $H_1$ is semi-simple and $R_1$ is the
radical,  we  may  assume  that  $G_1$ is essentially   simply  connected (Lemma \ref{surjectivemorphisms}(iii)),  without
affecting   the   hypotheses   or   the  conclusion   of   Proposition
\ref{arithmeticchevalley}.

Hence $G_1 = R_1 \rtimes H_1$ is a semi direct product.  Then clearly, every congruence subgroup of $G_1 $ contains a congruence subgroup of the form $\D \rtimes \Phi$ where $\D \subset R_1$ and $\Phi \subset H_1$ are congruence subgroups.  Similarly, write $G_2=R_2H_2$. Since $\varphi $ is easily seen
to map  $R_1$ onto $R_2$ and $H_1$  onto $H_2$, it is  enough to prove
the proposition for $R_1$ and $H_1$ separately.

We first recall  that if $G$ is a solvable linear algebraic group defined
over  $\Q$ then  the congruence  subgroup property  holds  for $G$,
i.e., every arithmetic subgroup of $G$ is a congruence subgroup (for a
reference    see   p.~108,   last    but   one    paragraph   of
\cite{Ra2} or \cite{Ch}).  Consequently, by  Lemma \ref{surjectivemorphisms} (ii),
the image of a congruence subgroup  in $R_1$ is an arithmetic group in
$R_2$ and hence a congruence subgroup. Thus we dispose of the solvable
case.

In the  case of  semi-simple groups, denote  by $H_2^*$ by  the simply
connected cover of  $H_2$. The map $\vp : H_1 \ra H_2$ lifts to a map from
$H_1$  to   $H_2^*$.  For  simply  connected   semi-simple  groups,  a
surjective map from  $H_1$ to $H_2^*$ sends a congruence subgroup to a congruence subgroup by Lemma 1.3 (iv).

We are thus reduced to the situation $H_1=H_2^*$ and $\vp: H_1\ra H_2$
is the simply connected cover of  $H_2$.

By our assumptions, $H_1$ is now connected, simply connected and semisimple.  We claim that for any non-trivial $\bbq$-simple factor $L$ of $H_1$, $L(\bbr)$ is not compact.  Otherwise, the image  of $\Gamma$, the arithmetic group, there is finite and as $\Ga$ is Zariski dense, so $H_1$ is not connected.  The strong approximation theorem (\cite[Theorem 7.12]{Pl-Ra}) gives now that $H_1(\bbq)$ is dense in $H_1(\bba_f)$.  So Proposition 2.1 can be applied to finish the proof of Proposition 0.1 in the case $\Ga$ is a congruence subgroup.

 We need to show that it is true also for the more general case when $\Ga$ is only congruence closed.  To this end let us formulate the following Proposition which is of independent interest.

\begin{prop}\label{pr2.3}  Let $\Ga \subseteq \GL_n(\bbz), G$ its Zariski closure and $Der = [G^0, G^0]$.  Then $\Ga$ is congruence closed if and only if $\Ga \cap Der$ is  a congruence subgroup of $Der$.
\end{prop}

\begin{proof}  If $G^0$ has no toral factors,  this is proved in \cite{Ve}, in fact,  in this case a congruence closed Zariski dense subgroup is a congruence subgroup.  (Note that this is stated there for general $G$, but the assumption that there is no toral factor was mistakenly omitted as the proof there shows.)

Now, if there is a toral factor, we can assume $G$ is connected, so $G^{ab} = V \times S$ where $V$ is unipotent and $S$ a torus.  Now $\Ga\cap [G, G]$ is Zariski dense and congruence closed, so it is a congruence subgroup by \cite{Ve} as before.  For the other direction, note that the image of $\Ga$ is $U \times S$, being solvable, is always congruence closed, so the Proposition follows.
\end{proof}

Now, we can end the proof of Proposition 0.1 for congruence closed subgroups by looking at $\varphi $ on $G_3 = \overline{\Ga}$ the Zariski closure of $\Ga$ and apply the proof above to $Der (G^0_3)$. It also proves Proposition \ref{congruenceimage}.

Of course, Proposition \ref{pr2.3}  is the general form of the following result from \cite{Ve} (based on \cite{Nori} and \cite{Weis}), which is, in fact, the core of Proposition \ref{pr2.3}. \begin{prop}  \label{noriconsequence} Suppose $\Gamma  \subset G(\Z)$
is Zariski  dense, $G$ simply  connected  and $\Gamma  $  a subgroup of $G(\Z)$ which is closed in the congruence topology. Then $\Gamma $ is itself a congruence
subgroup.
\end{prop}

\section{The Grothendieck closure}

\subsection{The Grothendieck Closure of a group $\G$}

\begin{defn} Let $\rho : \G \ra  \GL(V)$ be a representation of $\G$ on
a lattice  $V$ in a  $\Q$-vector space $V\otimes  \Q$.  Then we  get a
continuous      homomorphism      $\widehat{\rho}:     \widehat{\G}\ra
\GL(\widehat{V})$  (where,  for  a  group  $\Delta$,  $\widehat{\Delta}$
denotes its profinite completion) which extends $\rho$ .  \\

Denote by $Cl_{\rho }(\G)$ the subgroup of the profinite completion of
$\G$,  which   preserves  the  lattice   $V$:  $Cl_{\rho}(\G)=  \{g\in
\widehat{\G}: \widehat{\rho}(g)(V)\subset V\}$. In fact,  since $\det (\hat\rho(g)) = \pm 1$ for every $g \in \Ga$ and hence also for every $g \in \hat \Ga$, for $g \in Cl_g (\Ga), \,  \hat\rho (g) (V) = V$, and hence $Cl_\rho (\Gamma)$ is a subgroup of $\hat\Ga$.   We denote by $Cl(\G)$
the subgroup
\begin{equation}\label{eq3.1} Cl(\G)= \{g\in  \widehat{\G}: \widehat{\rho} (g)  (V)\subset V \quad
\forall  \quad lattices \quad  V \}.\end{equation}  Therefore, $Cl(\G)=\cap_{\rho}
Cl_{\rho}(\G)$ where $\rho$  runs through all integral representations
of the group $\G$.

Suppose  now that  $V$ is  any finitely  generated abelian  group (not
necessarily  a lattice  i.e.  not  necessarily torsion-free)  which is
also a  $\G$-module.  Then the torsion  in $V$ is  a (finite) subgroup
with finite exponent  $n$ say. Then $nV$ is  torsion free.  Since $\G$
acts on the  finite group $V/nV$ by a finite group via, say,  $\rho$, it
follows that $\widehat{\G}$  also acts on the finite  group $V/nV$ via
$\widehat{\rho}$.    Thus,    for   $g\in   \widehat{\G}$    we   have
$\widehat{\rho}(g)(V/nV)=V/nV$. Suppose  now that $g\in  Cl(\G)$. Then
$g(nV)=nV$  by  the  definition  of  $Cl(\G)$.   Hence
$g(V)/nV=V/nV$ for $g\in Cl(\G)$. This  is an equality in the quotient
group  $\widehat{V}/nV$.  This  shows that  $g(V)\subset  V+nV=V$ which
shows  that $Cl(\G)$  preserves {\it  all} finitely  generated abelian
groups $V$ which are $\G$ -modules.

By  $Cl_{\Z}(\G)$  we  mean  the  {\it Grothendieck  closure}  of  the
(finitely  generated)  group $\G$.   It  is  essentially  a result  of
\cite{Lub} that the Grothendieck  closure $Cl_{\Z}(\G)$ is the same as
the group $Cl(\G)$ defined  above (in \cite{Lub}, the group considered
was  the closure  with respect  to {\it  all} finitely  generated $\Z$
modules which  are also $\G$  modules, whereas we consider  only those
finitely generated $\Z$  modules which are $\G$ modules  and which are
torsion-free; the argument of the preceding paragraph shows that these
closures  are the  same). From  now on,  we identify  the Grothendieck
closure $Cl_{\Z}(\G)$ with the foregoing group $Cl(\G)$.
\end{defn}

\begin{notation}\label{BDG}   Let  $\G$  be  a  group, $V$  a  finitely  generated
torsion-free abelian  group which is  a $\G$-module and $\rho:  \G \ra
\GL(V)$ the corresponding $\G$-action. Denote by $G_{\rho}$ the Zariski
closure  of  the  image  $\rho(\Gamma  )$ in  $\GL(V\otimes  \Q)$,  and
$G_{\rho}^0$   its  connected  component   of  identity.    Then  both
$G_{\rho},G_{\rho}^0$ are  linear algebraic groups  defined over $\Q$,
and so is $Der_\rho = [G^0_\rho, G^0_\rho]$.

Let         $B = B_{\rho}(\G)$         denote         the         subgroup
$\widehat{\rho}(\widehat{\G})\cap  \GL(V)$.
Since the profinite topology of $\GL(\hat V) $ induces the congruence topology on $\GL(V), B_\rho(\Gamma)$ is the congruence closure of $\rho(\Gamma)$ in $\GL(V)$.

We denote by $D=D_{\rho}(\G)$ the intersection of $B$ with the derived
subgroup $Der_\rho = [G^0,G^0]$. We thus have an exact sequence
\[1 \ra D \ra B \ra A \ra  1, \] where $A= A_\rho (\Ga)$ is an extension of a finite
group $G/G^0$  by an abelian  group (the image  of $B\cap G^0$  in the
abelianization $(G^0)^{ab}$ of the connected component $G^0$).
\end{notation}

 \subsection{Simply Connected Representations}

\begin{defn}  \label{simplyconnectedrepresentations} We will  say that
$\rho$ is {\bf simply connected} if  the group $G = G_{\rho}$ is {\it essentially
simply connected}.  That is, if $U$  is the unipotent  radical of $G$,
the quotient $G^0/U$ is a product $H\times S$ where $H$ is semi-simple
and simply connected and $S$ is a torus.
\end{defn}

An easy     consequence     of     Lemma \ref{essentiallysimplyconnected}
is that      simply      connected
representations are closed under direct sums.

\begin{lemma} \label{SCdirectsum}  Let $\rho _1,\rho _2$  be two simply
connected representations  of an abstract group $\G$.  Then the direct
sum $\rho _1\oplus \rho _2$ is also simply connected.
\end{lemma}

We also have:

\begin{lemma}\label{surjectiveclosure} Let $\rho: \Gamma \ra \GL(W)$ be
a sub-representation of a  representation $\tau: \Gamma \ra \GL(V)$ such
that   both  $\rho,   \tau$  are simply connected.   Then  the   map
$r: B_{\tau}(\Gamma )\ra B_{\rho}(\Gamma )$ is surjective.
\end{lemma}

\begin{proof} The  image of $B_{\tau}(\G)$  in $B_{\rho}(\G)$ contains
the image  of $D_{\tau}$.  By Proposition   \ref{pr2.3}, $D_\tau$   is a congruence  subgroup of the
algebraic   group  $Der_{\tau}$.    The   map  $Der_{\tau}\ra
Der_{\rho}$   is a  surjective map between simply connected groups.   Therefore,  by   part  (iv)   of  Lemma
\ref{surjectivemorphisms},  the image  of $D_{\tau}$  is  a congruence
subgroup  $F$ of $D_{\rho}$.   Now, by Proposition \ref{pr2.3}, $D_\rho \cdot \rho(\Ga)$ is  congruence closed, hence equal to $B_\rho$ which is the congruence closure of $\rho(\Ga)$ and $B_\tau \to B_\rho$ is surjective.
\end{proof}

\subsection{Simply-Connected to General}

\begin{lemma}    \label{simplyconnectedsaturate}    Every   (integral)
representation  $\rho: \G  \ra  \GL(W)$ is  a  sub-representation of  a
faithful representation $\tau: \G \ra \GL(V)$ where $\tau$ is simply connected.
\end{lemma}

\begin{proof}  Let $\rho  : \G  \ra \GL(W)$  be a  representation.  Let
$Der$ be  the derived subgroup  of the identity component  of the
Zariski closure $H= G_\rho $  of $\rho (\G)$. Then, by Lemma \ref{surjectivemorphisms}(iii),  there exists  a map $H^* \ra
H^0$   with   finite   kernel      such   that $H^*$ is connected and    $H^*/U^*=
(H^*)_{ss}\times S^*$  where $H^*_{ss}$  is a simply  connected semi-simple
group.  Denote by $W_{\Q}$
the   $\Q$-vector   space   $W\otimes    \Q$.    By     Lemma
\ref{surjectivemorphisms}(i), $\rho :H^0  \ra \GL(W_{\Q})$ may be considered
as  a   sub-representation  of  a   faithful  representation  $(\theta,  E_{\Q}) $ of the covering group $H^*$. \\

By (ii) of Lemma  \ref{surjectivemorphisms}, the image of an arithmetic
subgroup of $H^*$ is an arithmetic  group of $H$.  Moreover, as $H(\bbz)$ is virtually torsion free, one may  choose a normal,    torsion-free arithmetic subgroup  $\D \subset  H(\Z)$ such
that  the map $H^*\ra  H^0 $ splits  over $\D$.   In particular,  the map
$H^*\ra H^0$  splits over  a normal subgroup $N$ of $\G$  of finite index.   Thus, $\theta$
may be considered as a representation of the group $N$. \\

Consider  the  induced  representation $Ind  _N^{\G}(W_{\Q})$.   Since
$W_{\Q}$   is   a   representation    of   $\G$,   it   follows   that
$Ind_N^{\G}(W_{\Q})=W_{\Q}  \otimes Ind_N^{\G}(triv_N)\supset W_{\Q}$.
Since, by  the first  paragraph of this  proof, $W _{\Q}\subset  E_{\Q}$ as
$H^*$ modules,  it follows that $W  _{\Q} \mid _N  \subset E_{\Q}$ and
hence   $W_{\Q}\subset  Ind   _N^{\G}(E_{\Q})=:V_{\Q}$.    Write  $\tau
=Ind_N^{\G}(E_{\Q})$ for the representation  of $\G$ on $V_{\Q}$.  The
normality of  $N$ in $\G$ implies that  the restriction representation
$\tau \mid _N$ is contained in a direct sum of the $N$-representations $n\to\theta(\gamma n\gamma^{-1})$ as $\gamma$ varies over the finite set $\Ga/N$.

Write  $G_{\theta \mid  _N}$  for  the Zariski  closure  of the  image
$\theta  (N)$.  Since  $G_{\theta  \mid _N}$  has  $H^*$ as its Zariski closure  and  the group $H^*_{ss}$  is simply
connected,  each $\theta  $ composed  with  conjugation by  $\g$ is  a
simply  connected  representation  of  $N$.   It  follows  from  Lemma
\ref{SCdirectsum} that $\tau \mid _N$ is simply connected. Since simple
connectedness of a representation is  the same for subgroups of finite
index, it follows that $\tau $, as a representation of $\G$, is simply
connected.

We have now proved that there exists $\G$-equivariant embedding of the
module $(\rho,W_{\Q})$ into $(\tau,  V_{\Q})$ where $W,V$ are lattices
in the $\Q$-vector spaces $W_{\Q},V_{\Q}$.  A basis of the lattice $W$
is  then a  $\Q$-linear  combination of  a  basis of  $V$; the  finite
generation of $W$  then implies that there exists  an integer $m$ such
that   $mW\subset  V$,  and   this  inclusion   is  an   embedding  of
$\G$-modules.   Clearly,  the  module   $(\rho,W)$  is   isomorphic  to
$(\rho,mW)$ the isomorphism given  by multiplication by $m$. Hence the
lemma follows.
\end{proof}

The following  is the main technical  result of this section, from which the  main results of this paper  are derived:

\begin{proposition}  \label{surjective}  The  group  $Cl(\G)$  is  the
inverse limit  of the groups $B_{\rho}(\G)$ where  $\rho$ runs through
simply connected  representations and $B_\rho(\Gamma)$ is the congruence closure of $\rho(\Gamma)$.  Moreover, if $\rho:  \G \ra \GL(W)$
is simply  connected, then the  map $Cl(\G)\ra B_{\rho} (\G)$  is
surjective.
\end{proposition}

\begin{proof}  Denote  temporarily by  $Cl(\G)_{sc}$  the subgroup  of
elements  of $\widehat{\G}$ which  stabilize the  lattice $V$  for all
{\it  simply connected}  representations $(\tau,  V)$. Let  $W$  be an
arbitrary  finitely generated  torsion-free  lattice which  is also  a
$\G$-module; denote by $\rho$ the action of $\G$ on $W$. \\

By   Lemma  \ref{simplyconnectedsaturate},   there  exists   a  simply
connected representation  $(\tau, V)$ which contains  $(\rho, W)$.  If
$g\in Cl(\G)_{sc}$,  then $\widehat{\tau}(g)(V)\subset V$;  since $\G$
is dense in $\widehat{G}$ and  stabilizes $W$, it follows that for all
$x\in       \widehat{\G}$,      $\widehat{\tau}(x)(\widehat{W})\subset
\widehat{W}$;     in    particular,     for     $g\in    Cl(\G)_{sc}$,
$\widehat{\rho}(g)(W)=   \widehat{\tau}(g)(W)\subset   \widehat{W}\cap
V=W$.  Thus, $Cl(\G)_{sc}\subset Cl(\G)$. \\

The group $Cl(\G)$  is, by definition, the set of  all elements $g$ of
the  profinite  completion  $\widehat{\G}$  which stabilize  all  $\G$
stable  torsion free lattices.  It follows  in particular,  that these
elements  $g$ stabilize  all  $\G$-stable lattices  $V$ associated  to
simply  connected  representations  $(\tau,V)$;  hence  $Cl(\G)\subset
Cl(\G)_{sc}$.     The   preceding    paragraph   now    implies   that
$Cl(\G)=Cl(\G)_{sc}$.    This   proves   the   first   part   of   the
proposition (see Equation \ref{eq0.2}).

We  can enumerate  all  the simply  connected integral  representations  $\rho$,
since $\G$ is finitely generated.  Write $\rho_1,\rho _2, \cdots, \rho
_n \cdots, $  for the sequence of simply  connected representations of
$\G$. Write $\tau _n$ for the direct sum $\rho _1\oplus \rho _2 \oplus
\cdots  \oplus \rho  _n$.  Then  $\tau _n\subset  \tau _{n+1}$  and by
Lemma  \ref{SCdirectsum} each  $\tau $  is simply  connected; moreover,
the  simply connected representation  $\rho _n$ is contained in $\tau
_n$. \\

By  Lemma \ref{surjectiveclosure},  it  follows that  $Cl(\G)$ is  the
inverse limit  of the {\it totally ordered  family} $B_{\tau _n}(\G)$;
moreover, $B_{\tau _{n+1}}(\G)$ maps  {\bf onto} $B_{\tau _n}(\G)$. By
taking inverse  limits, it follows  that $Cl(\G)$ maps {\it  onto} the
group $B_{\tau _n}(\G)$ for every  $n$.  It follows, again from Lemma
\ref{surjectiveclosure}, that every $B_{\rho _n}(\G)$ is a homomorphic
image  of $B_{\tau  _n}(\G)$ and  hence of  $Cl(\G)$. This  proves the
second part of the proposition.
\end{proof}

\begin{defn} Let $\G$ be a  finitely generated group. We say that $\G$
is  $FAb$ if  the abelianization  $\D  ^{ab}$ is  finite for  every
finite index subgroup $\D\subset \G$.
\end{defn}

\begin{cor}\label{closureinverse} If $\G$  is $FAb$  then  for every simply connected representation $\rho$, the congruence closure $B_\rho(\G)$ of $\rho(\G)$ is a congruence subgroup and $Cl(\G)$  is  an inverse
limit  over  a  totally  ordered  set  $\tau_n$  of  simply  connected
representations of  $\G$, of congruence  groups $B_n$ in  groups $G_n= G_{{\tau_n}}$
with $G^0_n$ simply connected. Moreover, the maps $B_{n+1}\ra B_n$ are
surjective. Hence the maps $Cl(\Ga)\ra B_n$ are all surjective.
\end{cor}

\begin{proof}  If   $\rho:  \G  \ra  \GL(V)$  is   a  simply  connected
representation, then  for a  finite index subgroup  $\G ^0$  the image
$\rho  (\G ^0)$  has  connected Zariski  closure,  and by  assumption,
$G^0/U=H\times S$  where $S$  is a torus  and $H$ is  simply connected
semi-simple. Since  the group $\G$  is $F_{Ab}$ it follows  that $S=1$
and  hence $G^0=Der  (G^0)$.  Now Proposition  \ref{noriconsequence} implies  that
$B_{\rho }(\G)$ is a congruence subgroup of $G_{\rho}(V)$.
The    Corollary    is   now    immediate    from   the    Proposition
\ref{surjective}. We take $B_n=B_{\tau _n}$ in the proof of the proposition.
\end{proof}

We can now prove Theorem 0.5. Let us first prove the direction claiming that the congruence subgroup property implies $Cl(\Ga) = \Ga$.  This was proved for arithmetic groups $\Ga$ by Grothendieck, and we follow here the proof in \cite{Lub} which works for general $\Ga$.  Indeed, if $\rho: \Ga \to \GL_n(\bbz)$ is a faithful simply connected representation such that $\rho(\Ga)$ satisfies the congruence subgroup property, then it means that the map $\hat\rho: \hat\Ga \to \GL_n (\hat\bbz)$ is injective.  Now $\rho \left(Cl(\Ga)\right) \subseteq \GL_n (\bbz) \cap \hat\rho (\hat\Ga)$, but the last is exactly the congruence closure of $\rho (\Ga)$.  By our assumption, $\rho(\Ga)$ is congruence closed, so it is equal to $\rho(\Ga)$.  So in summary $\hat\rho (\Ga) \subset \hat\rho \left(Cl (\Ga)\right)\subseteq\rho(\Ga) = \hat\rho(\Ga)$.  As $\hat\rho$ is injective, $\Ga = Cl (\Ga)$.

In the opposite direction: Assuming $Cl(\Ga) = \Ga$.  By the description of $Cl(\Ga) $ in (0.1) or in (3.1), it follows  that for every finite index subgroup $\Ga' $ of $\Ga$, $Cl(\Ga') = \Ga'$ (see \cite[Proposition 4.4]{Lub}).  Now, if $\rho $ is a faithful simply connected representation of $\Ga$, it is also such for $\Ga'$ and by Proposition 3.6, $\rho\left(Cl(\Ga')\right)$ is congruence closed.  In our case it means that for every finite index subgroup $\Ga'$, $\rho(\Ga')$ is congruence closed, i.e. $\rho(\Ga)$ has the congruence subgroup property.

\section{Thin Groups}

Let $\Ga$ be a finitely generated $\bbz$-linear group, i.e. $\Ga \subset \GL_n(\bbz)$, for some $n$.  Let $G$ be its Zariski closure in $\GL_n(\bbc)$ and $\D = G \cap \GL_n(\bbz)$.  We say that $\Ga$ is a \emph{thin} subgroup of $G$ if $[\D:\Ga] = \infty$, otherwise $\Ga$ is an arithmetic subgroup of $G$.  In general, given $\Ga$, (say, given by a set of generators) it is a difficult question to determine if $\Ga$ is thin or arithmetic.  Our next result gives,  still, a group theoretic characterization for the {\bf abstract} group $\Ga$ to be thin.  But first a warning: an abstract group can sometimes appear as an arithmetic subgroup  and sometimes as a thin  subgroup.  For example, the free group on two generators $F = F_2$ is a finite index subgroup of $\SL_2(\bbz)$, and so, arithmetic. But at the same time, by a well known result of Tits asserting that $\SL_n(\bbz)$ contains a copy of $F$ which is Zariski dense in $\SL_n$ [Ti];  it is also thin.  To be precise, let us define:

\begin{definition}\label{thindefinition}  A finitely generated $\bbz$-linear group $\Ga$ is called a {\bf thin group} if it has a faithful representation $\rho: \Ga \to \GL_n(\bbz)$ for some $n \in \bbz$, such that $\rho(\Ga)$ is of infinite index in $\overline{\rho(\Ga)}^Z \cap \GL_n (\bbz)$ where $\overline{\rho(\Ga)}^Z $ is the Zariski closure of $\Ga$ in $\GL_n$.  Such a $\rho$ will be called a thin  representation of $\Ga$.
\end{definition}

We have assumed that $i: \Ga \subset GL_n(\Z)$. Assume also, as we may (see Lemma
\ref{simplyconnectedsaturate}) that the representation $i$ is simply connected. By Proposition \ref{surjective}, the group  $Cl(\Ga)$ is the subgroup of $\widehat{\Ga}$ which preserves the lattices $V_n$ for a totally ordered set (with respect to the relation of being a sub representation)  of faithful
simply connected integral representations $(\rho _n,V_n)$ of $\Ga$ with the maps $Cl(\Ga) \ra B_n$ being surjective, where $B_n$ is the congruence closure of $\rho_n(\Gamma)$ in $GL(V_n)$.  Hence,    $Cl(\Ga)$ is the inverse limit (as $n$ varies) of the congruence closed subgroups $B_n$ and $\Ga$ is the inverse limit of the images $\rho _n(\Ga)$. Equip $B_n/\rho _n(\Ga)$ with the discrete topology.
Consequently, $Cl(\Ga)/\Ga$ is a closed subspace of the Tychonov product $\prod _n (B_{n}/\rho_n(\Ga))$. This is the topology on $Cl(\Ga)/\Ga$ considered in the following theorem.

\begin{thm}\label{thin compact}  Let $\Ga$ be a finitely generated $\bbz$-linear group, i.e. $\Ga \subset \GL_m (\bbz)$ for some $n$.  Then $\Ga$ is \emph{not} a thin group if and only if $\Ga$ satisfies both of the following two properties:

\begin{enumerate}[(a)]

\item $\Ga$ is an $FAb $ group (i.e. for every finite index subgroup $\Lambda $ of $\Ga$, $\Lambda/[\Lambda, \Lambda]$ is finite), and

\item The group $Cl(\Ga)/\Ga$ is compact

 \end{enumerate}
 \end{thm}

\begin{proof}  Assume first that $\Ga$ is a thin group.  If $\Ga$ is not $FAb$ we are done. So, assume $\Ga$ is $FAb$. We must now prove that $Cl(\Ga)/\Ga$ is not compact. We know that $\Ga$ has a faithful thin representation $\rho :\Ga \ra GL_n(\Z)$ which in addition, is simply connected.  This induces a surjective map (see Proposition \ref{surjective}) $Cl(\Ga)\ra B_\rho (\Ga)$ where $B_\rho (\Ga)$ is the  congruence closure of $\rho (\Ga)$ in $GL_n(\Z)$.
 As $\Gamma$ is $FAb, B_\rho(\Gamma) $ is a congruence subgroup, by Corollary \ref{closureinverse}.
 But as $\rho$ is thin,  $\rho (\Ga)$ has infinite index in $B_\rho (\Ga)$. Thus, $Cl(\Ga)/\Ga$ is mapped {\it onto} the discrete infinite quotient space $B_\rho (\Ga)/\rho (\Ga)$. Hence $Cl(\Ga)/\Ga$ is not compact.

Assume now $\Ga$ is not a thin group.  This implies that for every faithful integral representation $\rho(\Ga)$ is of finite index in its integral Zariski closure.  We claim that $\Ga/[\Ga, \Ga]$ is finite.  Otherwise, as $\Ga$ is finitely generated, $\Ga$ is mapped on $\bbz$.  The group $\bbz$ has a Zariski dense integral representation $\tau$ into $\mathbb{G}_a\times S$ where $S$ is a torus; take any integral matrix $g \in \SL_n(\bbz)$ which is neither semi-simple nor unipotent, whose semisimple part has infinite order.  Then both the unipotent and semisimple part of the Zariski closure $H$ of $\tau(\bbz)$ are non trivial and $H(\bbz)$ cannot contain $\tau(\bbz)$ as a subgroup of finite index since $H(\bbz)$ is commensurable to $\mathbb{G}_a (\bbz) \times S(\bbz)$ and both factors are non trivial and infinite.
The representation $\rho \times \tau$ (where $\rho$ is any faithful integral representation of $\Gamma$) will give a thin representation of $\Gamma$.  This proves that $\Ga/[\Ga, \Ga]$ is finite.  A similar argument (using an induced representation) works for every finite index subgroup, hence $\Ga$ satisfies $FAb$.

We now prove  that $Cl(\Ga)/\Ga$ is compact.  We already know that $\Ga$ is $FAb$, so by Corollary \ref{closureinverse}, $Cl(\Ga) = \underset{\leftarrow}{\lim} B_{\rho_n} (\Ga)$ when $B_n = B_{\rho_n} (\Ga)$ are congruence  groups with surjective homomorphisms $B_{n + 1} \to B_n$.  Note that as $\Ga$ has a faithful integral representation, we can assume that all the representations $\rho_n$ in the sequence are faithful and
\begin{equation}\label{inverse limit}  \Ga = \lim_{\stackrel{\longleftarrow}{n}} \rho_n(\Ga).\end{equation}
This implies that $Cl(\Ga)/\Ga = \lim\limits_{\stackrel{\longleftarrow}{n}} B_n/\rho_n (\Ga)$.  Now, by our assumption, each $\rho_n(\Ga)$ is of finite index in $B_n = B_{\rho_n} (\Ga)$.  So $Cl(\Ga)/\Ga$ is an inverse limit of finite sets and hence compact.
\end{proof}

\section{Grothendieck closure and super-rigidity}

Let $\Ga$ be a finitely generated group.  We say that $\Ga$ is \emph{integral super-rigid}  if there exists an algebraic group $G \subseteq \GL_m(\bbc)$ and an embedding $i:\Ga_0 \mapsto G$ of a finite index subgroup $\Ga_0$ of $\Ga$, such that for every integral representation $\rho:\Ga \to \GL_n(\bbz)$, there exists an algebraic representation $\tilde \rho: G\to\GL_n(\bbc)$  such that $\rho$ and $\tilde \rho$ agree on some finite index subgroup of $\Ga_0$. Note: $\Ga$ is integral super-rigid if and only if a finite index subgroup of $\Ga$ is integral super-rigid.

Example of such super-rigid groups are, first of all,  the irreducible (arithmetic) lattices in high rank semisimple Lie groups, but also the (arithmetic) lattices in the rank one simple Lie groups $Sp(n, 1)$ and $\bbf^{-20}_4 $ (see \cite{Mar}, \cite{Cor}, \cite{Gr-Sc}).  But \cite{Ba-Lu} shows that there are such groups which are thin groups.

Now, let $\Ga$ be a subgroup of $\GL_m(\bbz)$, whose Zariski closure is essentially simply connected.  We say that $\Ga$ satisfies the \emph{congruence}  \emph{subgroup}  \emph{property} (CSP) if the natural extension of $i:\Ga \to\GL_m(\bbz)$ to $\hat \Ga$, i.e. $\tilde i: \hat\Ga \to \GL_m(\hat\bbz)$ has finite kernel.

\begin{thm}\label{superrigid} Let $\Ga \subseteq \GL_m (\bbz)$ be a finitely generated subgroup satisfying $(FAb)$.  Then
\begin{enumerate}[{\rm(a)}]

\item $Cl(\Ga)/\Ga$ is compact if and only if $\Ga$ is an arithmetic group which is integral super-rigid.

\item $Cl(\Ga)/\Ga $ is finite if and only if $\Ga$ is an arithmetic group satisfying the congruence subgroup property.
\end{enumerate}
\end{thm}

\begin{remarks}\label{super}
\begin{enumerate}[{\rm(a)}]
\item  The finiteness of $Cl(\Ga)/\Ga$ implies,  in particular, its  compactness , so Theorem \ref{superrigid} recovers the well known fact (see \cite{BMS}, \cite{Ra2}) that the congruence subgroup property implies super-rigidity.

    \item As explained in \S 2 (based on \cite{Ser}) the simple connectedness is a necessary condition for the CSP to hold.  But by Lemma \ref{simplyconnectedsaturate}, if $\Ga$ has any embedding into $\GL_n(\bbz)$ for some $n$, it also has a simply connected one.
        \end{enumerate}
        \end{remarks}

We now prove Theorem \ref{superrigid}.

\begin{proof}: Assume first $Cl(\Ga)/\Ga$ is compact in which case, by Theorem \ref{thin compact}, $\Ga$ must be an arithmetic subgroup of some algebraic group $G$.  Without loss of generality (using Lemma \ref{simplyconnectedsaturate}) we can assume that $G$ is connected and simply connected, call this representation $\rho: \Ga \to G$.  Let $\theta$ be any other representation of $\Gamma$.

Let $\tau  =\rho \oplus \theta$  be the direct  sum. The
group  $G_{\tau}$ is  a subgroup  of $G_{\rho}\times  G_{\theta}$ with
surjective  projections.  Since  both   $\tau  $  and  $\rho  $  are
embeddings  of   the  group  $\G$,   and  $\G$  does  not   have  thin
representations, it  follows (from Corollary \ref{closureinverse}) that the  projection $\pi :  G_{\tau} \ra
G_{\rho}$ yields an isomorphism of the arithmetic groups $\tau (\G)\subset
G_{\tau}(\Z)$ and $\rho (\G)\subset G_{\rho}(\Z)$.

Assume, as we may, that $\G$ is torsion-free and $\G$ is an arithmetic
group.   Every  arithmetic  group  in $G_{\tau}(\Z)$  is  virtually  a
product   of  the   form   $U_{\tau}(\Z)\rtimes  H_{\tau}(\Z)$   where
$U_{\tau}$ and  $H_{\tau}$ are the unipotent and  semi-simple parts of
$G_{\tau}$ respectively (note that  $G_{\tau}^0$ cannot have torus as
quotient since $\G$ is $FAb$). Hence $\G\cap U_{\tau}(\Z)$ may also
be   described   as  the  virtually  maximal  normal nilpotent  subgroup   of
$\G$. Similarly for $\G\cap U_{\rho}(\Z)$. This proves that the groups
$U_{\tau}$  and  $U_{\rho}$ have  isomorphic  arithmetic groups  which
proves that $\pi: U_{\tau} \ra U_{\rho}$ is an isomorphism.  Otherwise $Ker(\pi)$, which is a $\bbq$-defined normal subgroup of $U_\tau$, would have an infinite intersection with the arithmetic group $\Ga\cap U_\tau$.

Therefore, the arithmetic groups in $H_{\tau}$ and $H_{\rho}$ are
isomorphic   and  the   isomorphism  is   induced  by   the  projection
$H_{\tau}\ra  H_{\rho}$.   Since  $H_{\rho}$  is simply  connected  by
assumption, and is  a factor of $H_{\tau}$, it  follows that $H_{\tau}$
is a product $H_{\rho}H$ where $H$ is a semi-simple group defined over
$\Q$ with  $H(\Z)$ Zariski dense in  $H$.  But the  isomorphism of the
arithmetic  groups in $H_{\tau}$  and $H_{\rho}$  then shows  that the
group $H(\Z)$  is finite  which means that  $H$ is  finite. Therefore,
$\pi: H_{\tau}^0\ra  H_{\rho}$  is an isomorphism and so  the  map
$G_{\tau}^0\ra  G_{\rho}$  is  also  an  isomorphism  since  it  is  a
surjective morphism  between groups of  the same dimension,  and since
$G_{\rho}$ is simply connected.

This proves that $\Ga$ is a super-rigid group.

In \cite{Lub}, it was proved that if $\Ga$ satisfies super rigidity in some simply connected group $G$, then (up to finite index) $Cl(\Ga)/\Ga$ is in 1-1 correspondence with $C(\Ga) = \Ker (\hat \Ga\to G(\hat\bbz))$.  This finishes the proof of both parts (a) and (b).
\end{proof}

\begin{remark}\label{csp and superrigid}  In the situation of Theorem \ref{superrigid}, $\Ga$ is an arithmetic group, satisfying super-rigidity. The difference between parts (a) and (b), is whether $\Ga$ also satisfies CSP.  As of now, there is  no known arithmetic group (in a simply connected group) which satisfies super-rigidity without satisfying CSP.  The conjecture of Serre about the congruence subgroup problem predicts that arithmetic lattices in rank one Lie groups fail to have CSP.  These include Lie groups like $Sp (n, 1)$ and $\bbf_4^{(-20)}$ for which super-rigidity was shown (after Serre had made his conjecture).  Potentially, the arithmetic subgroups of these groups can have $Cl(\Ga)/\Ga$ compact and not finite.  But (some) experts seem to believe now that these groups do satisfy CSP.  Anyway as of now, we do not know any subgroup $\Ga$ of $\GL_n(\bbz)$ with $Cl(\Ga)/\Ga$ compact and not finite.

\end{remark}

\end{document}